\documentclass[12pt]{amsart}
\usepackage{amsmath,amsfonts,amsthm,amscd,amssymb,stmaryrd}

\newtheorem{theorem}{Theorem}
\newcommand{\beq}{\begin{eqnarray}}
\newcommand{\eeq}{\end{eqnarray}}

\newcommand{\R}{\mathbb{R}}

\newcommand{\tr}{\mathrm{tr}}

\newcommand{\A}{\mathcal{A}}

\newcommand{\M}{\mathcal{M}}
\newcommand{\E}{\mathcal{E}}
\newcommand{\N}{\mathcal{N}}

\newcommand{\U}{\mathcal{U}}
\newcommand{\W}{\mathcal{W}}
\newcommand{\V}{\mathcal{V}}

\newtheorem{lemma}[theorem]{Lemma}

\newtheorem{definition}[theorem]{Definition}
\newtheorem{remark}[theorem]{Remark}

\newtheorem{convention}[theorem]{Convention}

\numberwithin{equation}{section}
\numberwithin{theorem}{section}
\usepackage{amsmath}
\usepackage{amssymb}
\usepackage{latexsym}

\begin{document}
\bibliographystyle{amsalpha}
\title{On the uniqueness of asymptotic limits of the Ricci flow}
\author{Antonio G. Ache}
\address{Mathematics Department, Princeton University, Fine Hall, Washington Road, Princeton New Jersey 08544-1000 USA }
\email{aache@math.princeton.edu}
\thanks{Research supported by a Postdoctoral Fellowship of the National Science Foundation, Award No. DMS-1204742}

\begin{abstract}
We consider a normalization of the Ricci flow on a closed Riemannian manifold given by the evolution equation $\partial_{t}g(t)=-2(Ric(g(t))-\frac{1}{2\tau}g(t))$  where $\tau$ is a fixed positive number.  Assuming that a solution for this equation exists for all time, and that the full curvature tensor and the diameter of the manifold are both uniformly bounded along the flow, we prove that up to the action of a 1-parameter family of diffeomorphisms, the flow converges to a unique shrinking gradient Ricci soliton using an idea of Sun and Wang developed in \cite{songwang10} to study the stability of the K\"{a}ler Ricci flow near a K\"{a}hler-Einstein metric. The method in \cite{songwang10} relies on the monotonicity of the Ricci flow along Perelman's $\W$-functional and a {\L}ojasiewicz-Simon inequality for the $\mu$-functional. Our result is an extension of the main theorem in \cite{sesum06}, where convergence to a unique soliton is proved assuming that the solution of the normalized Ricci flow in consideration is sequentially convergent to a soliton which satisfies a certain integrability condition.  

\end{abstract}
\maketitle

\section{Introduction}

Let $(M,g)$ be a closed Riemannian manifold, let $\tau>0$ be a fixed number and  let $g(t)$ be a solution of the normalized Ricci flow
\begin{align}\label{tauflow}
\left\{
\begin{array}{ll}
\partial_{t}g_{ij}&=-2(R_{ij}-\frac{1}{2\tau}g_{ij}),\\
g_{0}&=g.
\end{array}
\right.
\end{align}  
Note that equation \eqref{tauflow} is equivalent after a rescaling and a reparametrization to the Ricci flow. Assume in addition the following
\begin{enumerate}
\item\label{exists} The flow exists for all time $t\ge 0$,
\item\label{Rmbounded} $\|Rm\|_{C^{0}(M)}$ is uniformly bounded along the flow,
\item\label{diameter} The diameter of $(M,g(t))$ (denoted by $\mathrm{diam}(M,g(t))$) is uniformly bounded in $t\ge 0$.
\end{enumerate}

It was proved in \cite{sesum06} that under assumptions \eqref{exists}, \eqref{Rmbounded} and \eqref{diameter},  solutions of \eqref{tauflow} are \emph{sequentially convergent} to compact shrinking gradient Ricci solitons, i.e., for every sequence $\{t_{i}\}$ of times with $t_{i}\nearrow\infty$, there exists a subsequence $\{t^{'}_{i}\}$ such that $g(t^{'}_{i})$ converges to a compact shrinking gradient Ricci soliton $g_{0}$ in $C^{k,\gamma}$ for any $k> 1$. More precisely, when we say that $g_{0}$ is a shrinking gradient Ricci soliton, we mean that it satisfies
\begin{align}\label{solitoneq1}
Ric(g_{0})+\nabla^{2}f_{0}-\frac{1}{2\tau}g_{0}=0,
 \end{align}
for some smooth function $f_{0}$. In this case we say that $g_{0}$ is a \emph{sequential limit} of \eqref{tauflow}. An important question derived from the above mentioned result is whether the Ricci flow converges (at least up to the action of a 1-parameter family of diffeomorphisms) to a soliton at a \emph{definite rate} and not just sequentially. This question was  answered in the affirmative in \cite{sesum06} assuming that one of the sequential limits $g_{0}$ satisfies a certain \emph{integrability} condition. In the interest of brevity, we will only define this integrability condition in the Einstein case


\begin{definition}[Integrability in the Einstein case]\label{defintegrability} Let $g_{0}$ satisfy $Ric(g_{0})=\frac{1}{2\tau}g_{0}$. Let $\{g_{s}\}$ be a path of metrics with $g_{s}\big{|}_{s=0}=g_{0}$ and $\frac{d}{ds}g_{s}\big{|}_{s=0}=h$. We say that $g_{0}$ is integrable if every solution $h$ of the linearized deformation equation
\begin{align*}
\frac{d}{ds}(Ric(g_{s})-\frac{1}{2\tau}g_{s})\big{|}_{s=0}=0,
\end{align*}
can be obtained from a path of metrics $g_{s}$ satisfying 
\begin{align*}
Ric(g_{s})-\frac{1}{2\tau}g_{s}=0.
\end{align*}
\end{definition} 
In the Einstein case, the integrability condition means roughly speaking that the set of Einstein metrics has a Banach manifold structure near the metric $g_{0}$. The main result proved in \cite{sesum06} is the following
\begin{theorem}[\cite{sesum06}]\label{thmsesum}  Let $(M,g)$ be a compact Riemannian manifold of dimension $n\ge 3$ and let $g(t)$ be a solution of the normalized Ricci flow \eqref{tauflow} starting at $g$ which exists for all $t\ge 0$. Assume also that both $\|Rm(g(t))\|_{C^{0}(M)}$ and $\mathrm{diam}(M,g(t))$ are uniformly bounded along the flow. If one sequential limit of $g(t)$ is integrable,  then there exists $t_{0}>0$, a 1-parameter family of diffeomorphisms $\phi_{t}$ and a unique compact shrinking gradient Ricci soliton $g_{0}$ (unique up to diffeomorphisms) such that one has
\begin{align*}
\|\phi_{t}^{*}g(t)-g_{0}\|_{C^{k,\gamma}_{g(t)}}\le Ce^{-\delta t},
\end{align*} 
for some positive constants $\delta$ and $C$ and for any $k\ge 1$ and $0<\gamma<1$. 
\end{theorem}

In this paper we prove that the integrability condition in Theorem \ref{thmsesum} is \emph{not} needed  to obtain convergence of solutions of \eqref{tauflow} to a unique soliton. However,  our proof only allows us to conclude that the convergence of $g(t)$ to a soliton is in general only polynomial in time and not exponential in time as in Theorem \ref{thmsesum}. Our main result is the following

\begin{theorem}\label{mainthm} Let $(M,g)$ be a compact Riemannian manifold of dimension $n\ge 3$ and suppose that \eqref{tauflow} has a solution  $g(t)$ which exists for all $t\ge 0$. Assume also that both $\|Rm(g(t))\|_{C^{0}(M)}$ and $\mathrm{diam}(M,g(t))$ are uniformly bounded along the flow. Then, there exists a compact  shrinking gradient Ricci soliton $g_{0}$, a number $T>0$ and a 1-parameter family of diffeomorphisms $\{\phi_{t}\}$ such that for $t\ge T$   
\begin{align*}
\|\phi_{t}^{*}g(t)-g_{0}\|_{C^{k,\gamma}_{g(t)}(M)}\le C(t-T)^{-\beta},
\end{align*}
where $C,\beta$ are positive constants independent of $t$.  Moreover, the soliton $g_{0}$ is unique up to diffeomorphisms. 
\end{theorem}

For the proof of Theorem \ref{mainthm} we use a method developed in \cite{songwang10} which relies on the observation that the $\W$-functional introduced by Perelman in \cite{perelman1} is nondecreasing along the Ricci flow and that the $\mu$ functional (also introduced in \cite{perelman1}) satisfies a {\L}ojasiewicz-Simon type inequality for metrics which are sufficiently close to a shrinking gradient Ricci soliton. The method in \cite{songwang10} is at the same time based on ideas developed in \cite{ls83} where several unique asymptotic limit theorems for a class of nonlinear evolution equations are proved and where the {\L}ojasiewicz-Simon inequality was first derived. On the other hand, the proof of the uniqueness part in Theorem \ref{thmsesum} is based on an argument used in \cite{ct} to prove uniqueness of tangent cones at infinity for complete Ricci flat manifolds with Euclidean volume growth and quadratic curvature decay at infinity. In the same way that it is assumed in Theorem \ref{thmsesum} that one of the limiting solitons is integrable, the main uniqueness result in \cite{ct} assumes that one Ricci-flat tangent cone satisfies a certain integrability condition. Even though the results in \cite{ct} were to a certain extent also inspired by those in \cite{ls83}, no {\L}ojasiewicz-Simon inequality was known for the problem studied in \cite{ct} and hence the need of an integrability condition. In the next section we will make clear what we mean by {\L}ojasiewicz-Simon inequality and we will give a brief explanation of its use in our proof.

\subsection{Outline of the proof of Theorem \ref{mainthm}}\label{outline}  In \cite{perelman1}, it was shown that the Ricci flow (and consequently a flow like \eqref{tauflow}) can be thought of as the gradient flow of a functional called the $\W$-\emph{functional} whose critical points are precisely the shrinking gradient Ricci solitons.  The $\W$-functional is defined explicitly as
\begin{align}\label{Wfunctional}
\W(g,f,\tau)=\int_{M}\left[\tau\left(|\nabla f|^{2}+R_{g}\right)+(f-n)\right](4\pi\tau)^{-n/2}e^{-f}dV_{g},
\end{align} 
where $f$ is a smooth function defined on $M$ and $\tau>0$ is a real number. An important fact about $\W$ is that  it provides us with a \emph{monotonicity formula} along the Ricci flow. More precisely, if we consider the coupled system
\begin{align*}
\partial_{t}g(t)&=-2(Ric(g(t))+\nabla^{2}f(t)-\frac{1}{2\tau}g),\\
\partial_{t}f_{t}&=-\Delta f_{t}-R_{g(t)}-\frac{n}{2\tau},
\end{align*}
while keeping $\tau$ fixed we obtain
\begin{align}\label{Wmonotonicity}
\partial_{t}\W(g(t),f_{t},\tau)=2\tau\int_{M}\left|Ric(g(t))-\nabla^{2}f_{t}-\frac{1}{2\tau}g(t)\right|_{g(t)}^{2}(4\pi\tau)^{-n/2}e^{-f_{t}}dV_{g(t)}.
\end{align}
From equation \eqref{Wmonotonicity} we conclude that if $\partial_{t}\W(g(t),f_{t},\tau)=0$, then $g(t)$ is a shrinking gradient Ricci soliton. In fact, if we fix $\tau$ and let $(g_{0},f_{0})$ satisfy \eqref{solitoneq1} then 
$(f_{0},g_{0})$ is a critical point of the map $(g,f)\rightarrow\W(g,f,\tau)$. If we let $\{\psi_{t}\}_{t\ge 0}$ be the flow generated by the time-dependent vector field $-\nabla f_{t}$  then the metric $\psi_{t}^{*}g(t)$ is a solution of the normalized Ricci flow \eqref{tauflow}. With this in mind, we would like to reduce the proof of Theorem \ref{mainthm} to the proof of the following result in the theory of ordinary differential equations
\begin{theorem}[See \cite{lsimonbk2}]\label{proto} Let $\xi(t)\in\R^{n}$ be a $C^{1}$ solution in $t\in[0,\infty)$ of the equation 
\begin{align}\label{ODE}
\dot{\xi}(t)=-\nabla f(\xi(t)), 
\end{align}
where $f:\R^{n}\rightarrow\R$ is real analytic. Then $\xi_{0}=\lim_{t\rightarrow\infty}\xi(t)$ exists and for some $\beta>0$ we have
\begin{align}\label{definiterate}
|\xi(t)-\xi_{0}|=O(t^{-\beta}),~t\rightarrow\infty.
\end{align}
Furthermore, $\xi_{0}$ is a critical point of $f$. 
 \end{theorem} 
Note that solutions of  equation \eqref{ODE} satisfy an obvious monotonicity property, namely that $f(\xi(t))$ is non-increasing in $t$ since $\xi(t)$ is evolving in the direction of ``steepest descent" for $f$. The reason why the difference $\xi(t)-\xi_{0}$ in \eqref{definiterate} decays at a definite rate as $t$ tends to infinity is because in this case the function $f$ satisfies the \emph{{\L}ojasiewicz inequality}, a result which states that if $f:\R^{n}\rightarrow\R$ is real analytic and $y$ is a critical point of $f$  then there exists a number $\alpha\in (0,1)$ such that for $x$ sufficiently close to $y$ we have
\begin{align}\label{finiteloja}
|f(x)-f(y)|^{1-\frac{\alpha}{2}}\le C|\nabla f(x)|,
\end{align}   
where $C$ is a positive constant independent of $x$. Even though the classical {\L}ojasiewicz inequality was proved for functions in $\R^{n}$, it was shown by Leon Simon in \cite{ls83} that a similar inequality holds in several infinite-dimensional settings. In many of the variational problems studied by Simon, he considered  functionals which where strictly convex and real analytic. The general formulation of the infinite-dimensional {\L}ojasiewicz inequality in \cite{ls83} is as follows: let $\E$ be a functional defined by
\begin{align}\label{generalfunctional}
\E(u)=\int_{M}E(x,u(x),\nabla u(x))d\mathrm{vol},
\end{align}
where the function $E(x,y,z)$ is real analytic in $(x,y)$. Let $\nabla\E(u)$ be the gradient or differential of $\E$ at $u$ and assume that $u_{0}$ is a critical point of $\E$, i.e., $\nabla\E(u_{0})=0$. If the second variation of $\E$ at $u_{0}$ is elliptic (which holds under suitable convexity assumptions on the integrand $E$), then there exists a number $\alpha\in (0,1)$ such that  for $u$ sufficiently close to $u_{0}$ in $C^{2}(M)$ we have
\begin{align}\label{infiniteloja}
|\E(u)-\E(u_{0})|^{1-\frac{\alpha}{2}}\le C\|\nabla\E(u)\|_{L^{2}(M)},
\end{align}  
where $C$ is a positive constant that does not depend on $u$. Inequality \eqref{infiniteloja} is sometimes called \emph{{\L}ojasiewicz-Simon inequality} and its proof is based on a method called the \emph{Lyapunov-Schmidt reduction} which allows us to reduce the infinite dimensional case to the classical {\L}ojasiewicz inequality \eqref{finiteloja}. Note however, that in principle the $\W$ functional does not fit into this framework because it  is invariant under diffeomorphisms that act simultaneously on $g$ and $f$ and this prevents the functional from being strictly convex. Even if we were able to deal with this difficulty, for example by means of a gauge-fixing approach, the analysis of the $\W$ functional along the Ricci flow is complicated by the fact that we have to consider not only an evolution equation on the metric but also the evolution of a function that satisfies a backward heat equation. An alternative approach is to consider another functional introduced by Perelman called the $\mu$-\emph{functional},  which is defined by
\begin{align}\label{defmu}
\mu(g,\tau)=\inf_{\{f:\int_{M}4(\pi\tau)^{-n/2}dV_{g}=1\}}\{\W(g,f,\tau)\}.
\end{align} 
 Perelman proved in \cite{perelman1} that the infimum in \eqref{defmu} is always attained by a smooth function. The $\mu$-functional also plays an important role in controlling the soliton behavior of a metric along the Ricci flow but the consideration of \eqref{defmu} for our asymptotic limit problem introduces many technical difficulties. One important issue is that even though we know that a smooth minimizer $f$ always exists, it is not clear in general how this minimizer depends on the metric or even if it is unique for a given $g$.  This could create difficulties for computing the first and second variations  of the functional $g\rightarrow\mu(g,\tau)$ and for applying the Lyapunov-Schmidt reduction. It was proved by Sun and Wang in \cite{songwang10} that this uniqueness issue can be resolved by considering metrics near the shrinking soliton $g_{0}$ in \eqref{solitoneq1}. In other words, in a sufficiently small neighborhood of  $g_{0}$, the minimizer in \eqref{defmu} is unique and depends real-analytically on the metric. In this case, we can define a flow which can be thought of as the gradient flow of $\mu$ and which is equivalent to \eqref{tauflow} up to the action of a 1-parameter family of diffeomorphisms. Following \cite{songwang10} we will call this flow \emph{Modified Ricci flow} (it is also sometimes called \emph{Soliton Ricci Flow}, see for example \cite{npalisrf}). Furthermore, using a variant of the Ebin-Palais slice theorem one can fix a gauge where the second variation of $\mu$ at $g_{0}$ is elliptic, and then it is possible to adapt Leon Simon's approach to prove a {\L}ojasiewicz-type inequality for $\mu$ near $g_{0}$. If we now assume that $g_{0}$ is a sequential limit obtained in \cite{sesum06}, then the elements of a sequence $\{g(t_{i})\}$ which converges to $g_{0}$ will eventually lie in a neighborhood of $g_{0}$ where the Modified Ricci flow can be defined, and then the proof of Theorem \ref{mainthm} will be similar in spirit to the proof of Theorem \ref{proto}. This will prove that in our situation solutions of the modified Ricci flow converge to a shrinking soliton at a definite rate, which implies Theorem \ref{mainthm} after the use of a 1-parameter family of diffeomorphisms as suggested in the statement of Theorem \ref{mainthm}. A final comment about the {\L}ojasiewicz-Simon inequality is that the functional considered in \eqref{infiniteloja} depends on derivatives of order at most one on $u$ whereas the $\W$ and $\mu$ functionals involve second order derivatives on the metric. This difficulty was addressed in \cite{songwang10} and more recently in \cite{coldmini2012} for a more general class of functionals
 \footnote{The author recently learned that a {\L}ojasiewicz-Simon inequality has been proved for Perelman's $\lambda$-functional near a Ricci-flat metric on a closed manifold in \cite{muhas2013} following also the approach in \cite{coldmini2012}. This inequality had been proved in \cite{hasl2012} near Ricci-flat metrics that satisfy an integrability condition.}.
  
\subsection{Organization of the paper}  In Section \ref{monotonicitymu} we introduce a notion of regularity which was defined in \cite{songwang10} and that will allow us to give a precise definition of the Modified Ricci flow at least near a shrinking gradient Ricci soliton as outlined in Section \ref{outline} of this introduction. In Section \ref{uniqueness} we state in detail the {\L}ojasiewicz-Simon inequality for the $\mu$ functional and prove a stability property for solutions of the modified Ricci flow (Lemma \ref{stability}). This stability property is the key to reduce the proof of Theorem \ref{mainthm} to the proof of Theorem \ref{proto}. The proof of Theorem \ref{mainthm} will be given at the end of Section \ref{uniqueness}. Finally, in Appendix \ref{loj} we compute in detail the second variation of the $\mu$-functional at a shrinking gradient Ricci soliton which together with a gauge-fixing lemma (Lemma \ref{gaugefixing}) will be enough to use the results in \cite{coldmini2012} to give a proof of the {\L}ojasiewicz-Simon inequality for $\mu$ alternative to the one presented in \cite{songwang10}.

\subsection{Notation and conventions}
\begin{itemize}
\item For the rest of the paper, the metric $g_{0}$ will denote a compact shrinking gradient Ricci soliton satisfying \eqref{solitoneq1}.
\item $\M(M)$ will denote the space of smooth metrics on $M$.
\item $S^{2}(T^{*}M)$ will denote the space of smooth symmetric $(0,2)$ tensors on $M$.
\item We will frequently consider the linearization of functionals of the form 
\begin{align*}
\E:\M(M)\rightarrow\R,
\end{align*}
at a given metric $g\in\M(M)$. Given $h\in S^{2}(T^{*}M)$, consider a path of metrics $g_{\epsilon}$ for $\epsilon\in[0,1)$ satisfying $g_{\epsilon}\big{|}_{\epsilon=0}=g$ and $\frac{d}{d\epsilon}g_{\epsilon}\big{|}_{\epsilon=0}=h$. The linearization of $\E$ at $g$ in the direction of $h$ will be denoted by $\E^{'}_{g}(h)$ and is defined by
\begin{align*}
\E^{'}_{g}(h)=\frac{d}{d\epsilon}\E(g_{\epsilon})\big{|}_{\epsilon=0}.
\end{align*}  
\end{itemize}

\subsection{Acknowledgements} The author is grateful to Gang Tian for suggesting this problem and for his constant encouragement. The author is also indebted to Song Sun and Yuanqi Wang for explaining several technical points in their paper, to Jeff Viaclovsky for assistance in the formulation of Lemma \ref{gaugefixing}, to Nefton Pali for explaining several of his results, 
to Jeffrey Case for pointing out several useful references on Perelman's $\mu$-functional and to Yiyan Xu for many stimulating discussions. Finally, the author would like to thank Hans Joachim Hein, Luc Nguyen, Bing Wang, Reto M\"{u}ller, Aaron Naber, Haotian Wu and Davi Maximo for their interest and for sharing their insights on the Ricci flow.   

 \section{Monotonicity of the $\mu$-functional and the modified Ricci flow}\label{monotonicitymu}
 
 As mentioned in Section \ref{outline} of the introduction,  we would like to use the monotonicity of the $\W$-functional together with a gradient  (or {\L}ojasiewicz-Simon type) inequality  to prove convergence at a definite rate to a gradient Ricci soliton $g_{0}$ which was obtained as a sequential limit of a solution of \eqref{tauflow}. This is because we know by the monotonicity identity \eqref{Wmonotonicity} that $\W$ controls the soliton behavior of a metric along the Ricci flow. It turns out that in this case, it is more convenient to use Perelman's $\mu$-functional defined in \cite{perelman1} instead of the $\W$-functional. Recall that given a metric $g\in\M(M)$ and a number $\tau>0$, the $\mu$-functional is defined by  
 \begin{align*}
\mu(g,\tau)=\inf_{f\in \A_{0}}\{\W(g,f,\tau)\},
\end{align*}
where $\A_{0}$ is the class of functions
\begin{align*}
\A_{0}=\{f\in C^{\infty}(M):\int_{M}d\nu=1\},
\end{align*}
and $d\nu$ is the measure given by $d\nu=(4\pi\tau)^{-n/2}dV_{g}$. It was shown in \cite{perelman1}  that this infimum is always attained for some smooth function function $f$ and moreover, the minimizer satisfies the equation
\begin{align}\label{minimizermu}
\tau\left(2\Delta f-|\nabla f|^{2}+R\right)+(f-n)=\mu(g,\tau),
\end{align} 
(see for example \cite[Equation 12.16]{kleinerlott}). In principle the use of the $\mu$-functional introduces many technical problems, for it is not clear how the minimizer depends on the metric and therefore we cannot guarantee that we can obtain a sufficiently regular path $t\rightarrow (g(t),f_{t})$ where $g(t)$ is a solution of the Ricci flow equation and $f_{t}$ is a minimizer of $\W$, i.e., $\W(g(t),f_{t},\tau)=\mu(g(t),\tau)$. For this reason we consider the following notion of regularity introduced in \cite{songwang10}. Assume first without loss of generality that the function $f_{0}$ in \eqref{solitoneq1} satisfies $\int_{M}(4\pi\tau)^{-n/2}e^{-f_{0}}dV_{g_{0}}=1$.
\begin{definition}\label{regularmetric}
Let $k>1$ be an integer and let $\gamma$ be a number with $0<\gamma<1$.  A Riemannian metric $g\in\M(M)$ is said to be regular if there exists a $C^{k,\gamma}$ neighborhood $\mathcal{U}$ of $g$ in $\M(M)$ with the following properties

\begin{enumerate}
\item\label{uniquef} There is a unique function $\tilde{f}$ such that $\W(\tilde{g},\tilde{f},\tau)=\mu(\tilde{g},\tau)$.
\item\label{unitvolume} The function $\tilde{f}$ described in \eqref{uniquef}  satisfies $(4\pi\tau)^{-n/2}\int_{M}e^{-\tilde{f}}dV_{\tilde{g}}$=1.
\item\label{analytic} $\tilde{f}$ depends real analytically  on $\tilde{g}$.
\end{enumerate} 
A $C^{k,\gamma}$ neighborhood $\U$ of $g$ in the space $\M(M)$  satisfying \eqref{uniquef}, \eqref{unitvolume} and \eqref{analytic} is said to be a regular neighborhood of $g$.
\end{definition}
For the rest of this paper the parameters $k$ and $\gamma$ will be fixed.  A result proved in \cite{songwang10} is the following

\begin{lemma}[\cite{songwang10}]\label{lemregular} A compact shrinking gradient Ricci soliton $(M,g_{0})$ is regular.
\end{lemma}



The approach used in \cite{songwang10} to prove Lemma \ref{lemregular} is based on an application of the implicit function theorem to a fourth order operator together with several elliptic estimates and a logarithmic Sobolev inequality derived in \cite{rothaus}. In case a metric $g$ is in a regular neighborhood $\U$ of $g_{0}$, we use $P(g)$ to denote the unique function satisfying \eqref{uniquef}, \eqref{unitvolume} and \eqref{analytic} in Definition \ref{regularmetric}. It is then easy to compute the first variation of the $\mu$-functional at a regular metric using that $\W(\tilde{g},P(\tilde{g}),\tau)=\mu(\tilde{g},\tau)$ for all $\tilde{g}$ in a regular neighborhood of $g$. A simple computation shows that for any $h\in S^{2}(T^{*}M)$ we must have
\begin{align*}
\mu^{'}_{g}(h)=-\tau\int_{M}\langle Ric(g)+\nabla^{2}f-\frac{1}{2\tau}g,h\rangle_{g} d\nu,
\end{align*}     
where $f$ is the function associated to $g$ that satisfies conditions \eqref{uniquef},\eqref{unitvolume} and \eqref{analytic} in Definition \ref{regularmetric} (since the parameter $\tau$ is fixed throughout the paper, we often use $\mu$ to denote $\mu(\cdot,\tau)$). In particular,  if $g_{0}$ is a shrinking gradient Ricci soliton given by  \eqref{solitoneq1} then $g_{0}$ is a critical point of $\mu$ (i.e. of the functional $g\rightarrow \mu(g,\tau)$). We will use $\nabla\mu(g)$ to denote the gradient or differential of $\mu(\cdot,\tau)$ at a regular metric $g$ which is given explicitly by
\begin{align}\label{gradmufunctional}
\nabla\mu(g)=-\tau\left(Ric(g)+\nabla^{2}f-\frac{1}{2\tau}g\right).
\end{align} 

\begin{convention}\label{conventionP}\em{ From Lemma \ref{lemregular}, given a shrinking gradient Ricci soliton $g_{0}$, and a regular neighborhood $\U$ of $g_{0}$, we will assume for the rest of the paper that we have a fixed function $P:\U\rightarrow C^{k,\gamma}_{g_{0}}(M)$ such that the map $g\rightarrow P(g)$ satisfies conditions \eqref{uniquef}, \eqref{unitvolume} and \eqref{analytic} in Definition \ref{regularmetric}.
}
\end{convention}

\begin{remark}\label{meaninganalytic}\em{
The precise meaning of  ``real analytic dependence" in \eqref{analytic} in Definition \ref{regularmetric} is in the sense explained in  \cite[Section 2.7]{nirenbergtopics}. For example, the convention in \cite{nirenbergtopics} for defining real analytic maps between Banach spaces  is the following: let $X,Y$ be real Banach spaces and let $U\subset X$ be an open set. Let $\bar{X},\bar{Y}$ be the complexification of $X$ and $Y$ respectively,  then a map $f:U\subset X\rightarrow Y$ is said to be real analytic if  $f$ is the restriction of a holomorphic map $\bar{f}:V\subset \bar{X}\rightarrow \bar{Y}$ where $V$ is a neighborhood of $U$ in $\bar{X}$. Also, $\bar{f}:V\subset X\rightarrow \bar{Y}$ being holomorphic means that $\bar{f}$ is differentiable at every point $x\in V$.  Also if $f:V\subset \bar{X}\rightarrow \bar{Y}$ is continuous then $f$ is real analytic if and only if for every finite dimensional subspace $C\subset \bar{X}$ and every $y^{*}$ in the dual space of $\bar{Y}$ the map $y^{*}(f):C\cap U\rightarrow\mathbb{C}$ is holomorphic. 
In the same reference there is also a statement of the implicit function theorem in the real-analytic category. See also \cite[Sections 6]{coldmini2012}. 
}
\end{remark}

\subsection{The Modified Ricci Flow}

We now consider a construction which leads to a flow that can be thought of as the \emph{gradient flow} of the $\mu$-functional. Let $\U$ be a regular neighborhood of $g_{0}$, and let $g(t)$ be a solution of \eqref{tauflow} which exists for all time $t\ge 0$ and such that $g(T)\in \U$ for some $T>0$. For some $\epsilon>0$, all the metrics $g(t)$ with $t\in[T,T+\epsilon]$ are in $\U$, and by putting $f_{t}=P(g(t))$ we have $\W(g(t),f_{t},\tau)=\mu(g(t),\tau)$. Write any $t\in [T,T+\epsilon]$ as $t=T+s$ and let $\psi_{s}$ with $s\in[0,\epsilon]$ be the flow generated by the vector field $-\nabla f_{T+s}$ (time dependent). Note that the flow is defined as long as $g(T+s)\in\U$.  If we set $\tilde{g}(s)=\psi^{*}_{s}g(T+s)$ then the metric $\tilde{g}(s)$ satisfies the system
\begin{align}
\partial_{s}\tilde{g}(s)&=-2\left(Ric(\tilde{g}(s))+\nabla^{2}\tilde{f}_{s}-\frac{1}{2\tau}\tilde{g}(s)\right),\label{solitongradflow}\\
\tilde{g}(0)&=g(T),\label{initialmetric}
\end{align}
where $\tilde{f}_{s}=\psi_{s}^{*}f_{T+s}$ and where the Hessian $\nabla^{2}\tilde{f}_{s}$ in \eqref{solitongradflow} is taken with respect to the metric $\tilde{g}(s)$. Note that if we apply this construction to a metric $g_{0}$ satisfying \eqref{solitoneq1}, i.e., if we take $\tilde{g}(0)=g_{0}$ then $\{\tilde{g}(s)\}$ exists for all $s\ge 0$ and moreover $\tilde{g}(s)=g_{0}$ for all $s\ge0$. For the flow $\{\tilde{g}(s)\}$ to be the gradient flow of the $\mu$-functional, there are two conditions that need to be verified, namely
\begin{enumerate}
\item[(a)] For each time $s$ such that $\tilde{g}(s)$ is defined, the metric $\tilde{g}(s)$ is regular,
\item[(b)] In  \eqref{solitongradflow}, $\tilde{f}_{s}$  is the unique function associated to $\tilde{g}(t)$ such that conditions (\ref{uniquef}), (\ref{unitvolume}) and (\ref{analytic}) in Definition \ref{regularmetric} are satisfied.
\end{enumerate}
Let $\phi:M\rightarrow M$ be a $C^{\infty}$ diffeomorphism, then recall that the $\W$ functional is diffeomorphism invariant in the sense that for a metric $g$ and a smooth function $f$ we have $\W(\phi^{*}g,\phi^{*}f,\tau)=\W(g,f,\tau)$. In case $g_{0}$ is a shrinking gradient Ricci soliton, the metric $\phi^{*}g_{0}$ is also a shrinking soliton and therefore, $\phi^{*}g_{0}$ is regular. In particular, if $\U$ is a sufficiently small regular neighborhood of $g_{0}$ and $\phi$ is sufficiently close to the identity in the $C^{k+1,\gamma}$ topology then the image of $\U$ under $\phi$ is a regular neighborhood of $\phi^{*}g_{0}$ and in consequence, if $g\in \U$ and $f$ is such that $\W(g,f,\tau)=\mu(g,\tau)$, then $\phi^{*}f$ is the only function described in Definition \ref{regularmetric} satisfying $\W(\phi^{*}g,\phi^{*}f,\tau)=\mu(\phi^{*}g,\tau)=\mu(g,\tau)$. Using this observation together with standard existence theory for solutions of ordinary differential equations, we see that we can guarantee (a) and (b) at least for short time.  Note that the functions $\tilde{f}(s)$ above constructed are given by $P(\tilde{g}(s))$ as in Convention \ref{conventionP}. The above construction motivates the following definition

\begin{definition}\label{defmodRF}
A one parameter family of metrics $\{\tilde{g}(t)\}$ is said to be a solution of the Modified Ricci Flow in the interval $[0,T)$ (possibly with $T=\infty$) starting at $g$, if there exists a regular neighborhood $\U$ of the shrinking Ricci soliton $g_{0}$ such that $\tilde{g}(t)\in\U$ for $t\in [0,T)$ and solves the equation
\begin{align}\label{ModRF}\left\{
\begin{array}{ll}
\partial_{t}\tilde{g}(t)=-2\left(Ric(\tilde{g}(t))+\nabla^{2}\tilde{f}_{t}-\frac{1}{2\tau}\tilde{g}(t)\right),\\
\tilde{g}(0)=g,
\end{array}
\right.
\end{align}
where $\tilde{f}_{t}=P(\tilde{g}(t))$ as in Convention \ref{conventionP}.
\end{definition}

Let $\tilde{g}(t)$ be a solution of the modified Ricci flow \eqref{ModRF}, then by using the first variation of the  $\mu$-functional, we obtain the following monotonicity identity along the modified Ricci flow
\begin{align*}
\frac{d}{dt}\mu(\tilde{g}(t),\tau)=2\tau\int_{M}\left| Ric(\tilde{g}(t))+\nabla^{2}\tilde{f}_{t}-\frac{1}{2\tau}\tilde{g}(t)\right|_{\tilde{g}(t)}^{2}d\nu_{t},
\end{align*}
 where $d\nu_{t}$ is the measure given by
 \begin{align*}
 d\nu_{t}=(4\pi\tau)^{-n/2}e^{-\tilde{f}_{t}}dV_{\tilde{g}(t)}.
 \end{align*}

We summarize the above remarks in the following

\begin{lemma}\label{existence}  Let $g(t)$ be a solution of \eqref{tauflow} which exists for all time $t\ge 0$. If  $\|g(0)-g_{0}\|_{C^{k+4,\gamma}_{g_{0}}}<\delta$ for $\delta>0$ sufficiently small, then there exists an interval $[0,T]$ with $T>0$ such that there exists a solution $\tilde{g}(t)$ of the modified Ricci flow \eqref{ModRF} and for all $t\in[0,T]$ we have $\|\tilde{g}(t)-g_{0}\|_{C^{k,\gamma}_{g_{0}}}<\delta$. Here by $C^{m,\gamma}_{g_{0}}$ we mean that the H\"{o}lder norm is taken with respect to the metric $g_{0}$.
\end{lemma}

\begin{proof} By the results in \cite[Chapter 3]{chowknopf}, there exists an interval $[0,T]$ with $T>0$ such that there exists a unique solution $\hat{g}(t)$ with $\hat{g}(0)=g(0)$ of the Ricci-DeTurck flow associated to the normalization \eqref{tauflow} and such that $\|\hat{g}(t)-g_{0}\|_{C^{k+4,\gamma}_{g_{0}}}<\delta$  for all $t\in[0,T]$. This is because the Ricci-DeTurck flow is a strictly parabolic quasilinear system which can be solved locally in time by means of a fixed-point approach, and hence the solution $\hat{g}$ depends continuously on the initial data. From $\hat{g}$ we obtain a unique solution of \eqref{tauflow} with $g(0)$ prescribed, and without loss of generality we may assume that $T$ is such that for all $t\in[0,T]$ the solution $g(t)$ to \eqref{tauflow} satisfies $\|g(t)-g_{0}\|_{C^{k+2,\gamma}_{g_{0}}}<\delta$. We may also assume that $\delta$ is small enough so that all metrics  $g(t)$ with $t\in [0,T]$ are regular. We can then follow the construction outlined at the beginning of this subsection to obtain a solution $\tilde{g}(t)$ to the modified Ricci flow \eqref{ModRF} with initial condition $g(0)$ for $t\in[0,T]$ and  we may assume that $T>0$ is such that $\|\tilde{g}(t)-g_{0}\|_{C^{k,\gamma}_{g_{0}}}<\delta$ for $t\in [0,T]$. Note that we need to assume that the initial metric  satisfies $\|g(0)-g_{0}\|_{C^{k+4,\gamma}_{g_{0}}}<\delta$ because in passing from the Ricci-DeTurck flow associated to \eqref{tauflow} to a solution of \eqref{tauflow} there is a loss of two derivatives by the action of a 1-parameter family of diffeomorphisms. Analogously, there is a loss of two derivates when passing from \eqref{tauflow} to the modified Ricci flow.

\end{proof}

\begin{remark}\em{
In the statement of Lemma \eqref{existence}, the uniqueness of solutions of \eqref{ModRF} is not addressed. This technical issue is the subject of the following subsection
}
\end{remark}

\subsection{Canonical solutions of the Modified Ricci flow} 

As seen in the construction motivating Definition \ref{defmodRF}, the existence of solutions for the modified Ricci flow presents many technical issues. Recently, a version of this flow called \emph{Soliton Ricci Flow} has been studied in great detail in  \cite{npaliwvar} and \cite{npalisrf}, but our approach is different since it only concerns regular metrics. A basic question about the modified Ricci flow is if we can formulate a simple long time existence criterion as the one in \cite{ham82} for the Ricci flow. For this purpose it would be useful to know that the solution of the modified Ricci flow above constructed is unique, but to prove a general uniqueness result seems to present many technical difficulties. Instead of proving  uniqueness of solutions of the modified Ricci flow in full generality, we show that it is always possible to define a smooth canonical solution which has a very useful extension property. Before proving our next result, we introduce the following notation: if $\delta>0$ we will use $\U^{k,\gamma}_{\delta}$ to denote a neighborhood of $g_{0}$ of the form
\begin{align}\label{holderneighborhood}
\U^{k,\gamma}_{\delta}=\{g\in\M:\|g-g_{0}\|_{C^{k,\gamma}_{g_{0}}(M)}<\delta\}.
\end{align}
We will assume that $\delta$ is small enough so that  $\U^{k,\gamma}_{\delta}$ is a regular neighborhood of $g_{0}$. We now show that there is a canonical way of choosing \emph{smooth} solutions of the modified Ricci flow.      


\begin{lemma}\label{canonicallemma} Let $g(t)$ be a solution of \eqref{tauflow} which exists for all time $t\ge 0$ and let $\tilde{g}(t)$ be a solution of \eqref{ModRF} with $\tilde{g}(0)=g(0)$ which exists for all $t$ in the interval $[0,T]$ for some $T>0$ (in particular all metrics $\tilde{g}(t)$ are regular for $0\le t\le T$). Then for all $t\in [0,T]$ there exists a unique function $f_{t}$  such that 
\begin{align}
\W(g(t),f_{t},\tau)&=\mu(g(t),\tau)\label{minimizing}\\
\int_{M}(4\pi\tau)^{-n/2}e^{-f_{t}}dV_{g(t)}&=1\label{unitintegral}.
\end{align}
Moreover, if $g(0)$ is in a regular neighborhood $\U^{k,\gamma}_{\delta}$ of a shrinking soliton $g_{0}$ and if $\psi_{t}$ is the flow generated by the vector field $-\nabla\psi_{t}$ for $0\le t\le T$ then for some $T^{'}$ with $0<T^{'}\le T$ the metric $\bar{g}(t)=-\psi_{t}^{*}g(t)$ is a solution of \eqref{ModRF} and $\bar{g}(t)\in\U^{k,\gamma}_{\delta}$ for $t\in[0,T^{'}]$ and all the metrics $\bar{g}(t)$ are smooth.
\end{lemma}

\begin{remark}\em{
 Note that the statement of uniqueness involving equations \eqref{minimizing} and \eqref{unitintegral} is very similar to the uniqueness property in Definition \ref{regularmetric}, however, in this case we are not a priori assuming that $g(t)$ is in a regular neighborhood of $g_{0}$.
 }
\end{remark}

\begin{proof}
In order to prove the uniqueness of the functions $f_{t}$, we let $\hat{f}_{t}$ be any other smooth function satisfying \eqref{minimizing}, \eqref{unitvolume}, then we note that for every $t\in[0,T]$ there exists a diffeomorphism $\phi_{t}$ such that $\phi^{*}_{t}\tilde{g}(t)=g(t)$, and by the diffeomorphism invariance of the $\W$ functional we have
\begin{align*}
\W(\tilde{g}(t),(\phi^{-1}_{t})^{*}f_{t},\tau)=\mu(\tilde{g}(t),\tau)=\W(\tilde{g}(t),(\phi^{-1}_{t})^{*}\hat{f}_{t},\tau),
\end{align*}   
and since all metrics $\tilde{g}(t)$ are regular for $t\in [0,T]$, we have $(\phi^{-1}_{t})^{*}f_{t}=(\phi^{-1}_{t})^{*}\hat{f}_{t}$ and therefore $f_{t}=\hat{f}_{t}$.  The rest of the lemma follows easily from the proof of Lemma \ref{existence}, since for some $T^{'}>0$ (in principle $T^{'}\le T$) we can guarantee that all metrics $\hat{g}(t)=\psi^{*}_{t}g(t)$, with $t\in[0,T^{'}]$ are regular. 
\end{proof}
The solution to \eqref{ModRF} obtained in Lemma \ref{canonicallemma} is called the \emph{Canonical Solution of the Modified Ricci Flow}.  Note that the solution $\bar{g}(t)$ in Lemma \ref{canonicallemma} is smooth since the functions $\bar{f}(t)$ are smooth. For this canonical solution we have the following extension result.  
 
 
 

\begin{lemma}\label{criterion} Suppose that $g(t)$ is a solution of the Ricci flow which exists for all time $t\ge 0$ and suppose that $g(0)\in\U^{k,\gamma}_{\delta}$. Suppose that the canonical solution $\tilde{g}(t)$ of the Ricci flow with $\tilde{g}(0)=g(0)$ is defined in $[0,T)$ and for all $t\in [0,T)$ we have $\tilde{g}(t)\in\U^{k,\gamma}_{\delta}$. If $\limsup_{t\nearrow T}\|\tilde{g}(t)-g_{0}\|_{C^{k,\gamma}_{g_{0}}}<\delta$ then there exists $\epsilon>0$ such that $\tilde{g}(t)$ is defined and smooth in $[0,T+\epsilon]$ and $\tilde{g}(t)\in\U^{k,\gamma}_{\delta}$ for all $t\in[0,T+\epsilon]$. 
\end{lemma}

\begin{proof} We first show that $\tilde{g}(t)$ extends to $[0,T]$ as a family of smooth metrics. Recall that for $t\in [0,T)$ we can obtain $\tilde{g}(t)$ as $\tilde{g}(t)=\psi^{*}_{t}g(t)$ where  $\psi_{t}$ is the one parameter family of diffeomorphisms generated by the time-dependent vector fields $\nabla f_{t}$ and $f_{t}$ is the unique smooth function that satisfies \eqref{minimizing} and \eqref{unitintegral}. In particular, the functions $f_{t}$ satisfy
\begin{align}\label{potentialeq}
\Delta_{g(t)}f_{t}-\frac{1}{2}|\nabla f_{t}|^{2}+\frac{1}{2}R_{g(t)}+\frac{1}{2\tau}(f_{t}-n)=\frac{1}{2\tau}\mu(g(t),\tau),
\end{align} 
and since $g(t)$ is a solution of \eqref{tauflow} which is defined and smooth in $[0,T]$, we have bounds of the form
\begin{align*}
\sup_{[0,T]}\|\nabla^{l}Rm_{g(t)}\|_{C^{0}_{g_{0}}}\le C(l),
\end{align*} 
where $C(l)$ is a positive constant. If we let $\Lambda_{t}$ be the Sobolev constant of $g(t)$ we also have a uniform bound of the form
\begin{align*}
\sup_{[0,T]}\Lambda_{t}\le \Lambda,
\end{align*} 
for some positive constant $\Lambda$. 
It follows from \eqref{potentialeq} and from the results in \cite[Section 1]{rothaus} or \cite[Lemma 3.5]{songwang10} that if we let $M_{l}=\sup_{[0,T)}\|Rm(g(t))\|_{C^{l}_{g_{0}}}$ for $l\ge 0$ an integer and  $\zeta=\sup_{[0,T]}\mu(g(t),\tau)$,  then we have bounds of the form
\begin{align*}
\sup_{[0,T)}\|\nabla^{l}f_{t}\|_{C^{0}_{g_{0}}}\le C(M_{l},\Lambda,\zeta),
\end{align*} 
and therefore, for the flow $\psi_{t}$ and any integer $l\ge 0$ we have estimates of the form
\begin{align*}
\sup_{[0,T)}\|\nabla^{l}\psi_{t}\|_{C^{0}_{g_{0}}}\le C^{'}(M_{l+1},\Lambda,\zeta,T),
\end{align*}
which implies that $\tilde{g}(t)=\psi^{*}_{t}g(t)$ is defined for $t=T$ since
\begin{align*}
\partial_{t}\tilde{g}(t)=-2\left(Ric(\tilde{g}(t))+\nabla^{2}\tilde{f}_{t}-\frac{1}{2\tau}\tilde{g}(t)\right)~\text{for}~0\le t<T,
\end{align*}
and $\tilde{f}_{t}$ depends real analytically on $\tilde{g}(t)$. From the condition 
\begin{align*}
\limsup_{t\nearrow T}\|\tilde{g}(t)-g_{0}\|_{C^{k,\gamma}_{g_{0}}}<\delta,
\end{align*}
 it also follows that $\tilde{g}(T)\in\U^{k,\gamma}_{\delta}$ and hence $\tilde{g}(T)$ is regular so there exists a  unique function $\tilde{f}_{T}$ satisfying \eqref{uniquef},\eqref{unitvolume} and \eqref{analytic} in Definition \ref{regularmetric}. From the proof of Lemma \ref{existence}, it follows that we can extend the canonical solution $\tilde{g}(t)$ to some interval $[0,T+\epsilon]$ while ensuring $\tilde{g}(t)\in\U^{k,\gamma}_{\delta}$ for $t\in[0,T+\epsilon]$.      
\end{proof} 
 
 We close this section by pointing out that regardless of the notion of regularity introduced above, we have the following property for the $\mu$-functional which was proved in \cite{sesum06}:
 \begin{lemma}[\cite{sesum06}]\label{crucialmu} Given the assumptions in Theorem \ref{mainthm} for $g(t)$ satisfying \eqref{tauflow}, the limit $\displaystyle{\lim_{t\rightarrow\infty}\mu(g(t),\tau)}$ exists and is finite for fixed $\tau$. In particular, if the soliton $g_{0}$ is obtained as a sequential limit of \eqref{tauflow} we have $\mu(g(t),\tau)\le\mu(g(0),\tau)$ for all $t\ge 0$ and furthermore $\displaystyle{\lim_{t\rightarrow\infty}\mu(g(t),\tau)=\mu(g_{0},\tau)}$.
 \end{lemma}
 
 \begin{proof}
 See Claim 3.5 and Lemma 3.6 in \cite{sesum06}. The proof is based on a continuity property of $\mu(\cdot,\tau)$ and the monotonicity of the $\W$-functional along the Ricci flow.  \end{proof}

 \section{Stability, convergence and uniqueness}\label{uniqueness}
 
 In this section we prove our main theorem  using a {\L}ojasiewicz-Simon inequality which we now state
 
 \begin{lemma}\label{lemmaloja} Let $g_{0}$ satisfy \eqref{solitoneq1}, then there exists a number $0<\alpha<1$ and a regular neighborhood $\U$ of $g_{0}$ such that for any $g\in\U$ the following inequality holds
 \begin{align}\label{loja}
 \left|\mu(g_{0},\tau)-\mu(g,\tau)\right|^{1-\alpha/2}\le C\|\nabla\mu(g)\|_{L^{2}_{g}(M)},
 \end{align} 
 where $\nabla\mu(g)$ is given by \eqref{gradmufunctional}, $C$ is a positive constant (independent of $g$) and where $L^{2}_{g}(M)$ denotes the $L^{2}$ norm on $(0,2)$-tensors given by
 \begin{align*}
 \|h\|^{2}_{L^{2}_{g}(M)}=\int_{M}\|h\|_{g}^{2}(4\pi\tau)^{-n/2}e^{-f}dV_{g},
 \end{align*}
where $f$ satisfies \eqref{uniquef}, \eqref{unitvolume}, \eqref{analytic} in Definition \ref{regularmetric} and $\|\cdot\|_{g}$ is the standard norm on $(0,2)$ tensors induced by the metric $g$
\end{lemma} 

\begin{proof} See \cite[Lemma 3.1]{songwang10} or Appendix \ref{loj}.
\end{proof}

\begin{remark}\label{equivalent} \em{Since $g$ is in a regular neighborhood $\U$ of $g_{0}$, by assuming that $\U$ is small enough we conclude that all norms $\|\cdot\|_{L^{2}_{g}(M)}$ are equivalent, so inequality \eqref{loja} holds even if we replace the norm $\|\cdot\|_{L^{2}_{g}(M)}$ with the norm $\|\cdot\|_{L^{2}_{g_{0}}(M)}$ possibly with a different multiplicative constant $C$ which is uniformly bounded for metrics in $\U$.
}
\end{remark}
 
Recall the notation introduced in \eqref{holderneighborhood}. We will assume that $\delta$ is small enough so that $\U^{k,\gamma}_{\delta}$ is a regular neighborhood of $g_{0}$. For simplicity, along the modified Ricci flow we will use $L^{2}_{t}(M)$ instead of $L^{2}_{\tilde{g}(t)}(M)$ and we will write $C^{k,\gamma}_{t}(M)$ to denote $C^{k,\gamma}_{\tilde{g}(t)}(M)$. If $l$ is a positive integer, we will write $\|\cdot\|_{W^{l,2}_{t}(M)}$ to denote the Sobolev norm given by
\begin{align*}
\|h\|_{W^{l,2}_{t}(M)}=\sum_{j=0}^{l}\|\nabla^{j}h\|_{L^{2}_{t}(M)}.
\end{align*}

The next Lemma is a standard stability inequality for solutions of gradient flow equations which was originally derived in \cite{ls83}. More recently, a similar estimate was used in \cite{schulze11} to prove uniqueness of compact tangent flows in the mean curvature flow. We include its proof for convenience.   

\begin{lemma}[Stability Inequality]\label{stability} Let $\tilde{g}(t)$ be a canonical solution of the modified Ricci flow on $[t_{0},t_{1}]$ such that $\tilde{g}(t)\in\U^{k,\gamma}_{\delta}$ (a regular neighborhood of $g_{0}$) with $\mu(\tilde{g}(t))\le\mu(g_{0})$ for all $t\in[t_{0},t_{1}]$, then there exists $\theta>0$ such that for any $t_{0}\le\tau_{0}\le\tau_{1}\le t_{1}$ we have the inequality
\begin{align*}
\|\tilde{g}(\tau_{1})-\tilde{g}(\tau_{0})\|_{C^{k,\gamma}_{g_{0}}(M)}\le C(\mu(g_{0})-\mu(\tilde{g}(t_{0})))^{\theta}.
\end{align*}
 \end{lemma}

\begin{proof} 
For the rest of the proof $C$ will denote a positive constant that may change from line to line and which depends only on $g_{0}$, $\tau$, $k$, $\gamma$, on bounds on the curvature tensor and the diameter along \eqref{tauflow}, on the parameter $0<\alpha<1$ in Lemma \ref{lemmaloja} and the multiplicative constant in inequality \eqref{loja}.  We will also indicate the dependence of $C$ on additional parameters that may appear, for example Sobolev constants (while keeping the dependence on the aforementioned parameters).  Let $m$ be a positive integer satisfying
\begin{align}\label{holderparameter}
\frac{1}{2}\le\frac{m-k-\gamma}{n}.
\end{align}
By the interpolation inequality for tensors (see for example \cite{ham82}), there exists $\eta\in(\frac{1-\alpha}{1-\alpha/2},1)$ where $\alpha$ is as in the {\L}ojasiewicz-Simon inequality in Lemma \ref{lemmaloja},  such that
\begin{align*}
\|\partial_{t}\tilde{g}(t)\|_{W^{m,2}_{t}}\le C(\eta,m)\|\partial_{t}\tilde{g}(t)\|_{L^{2}_{t}}^{\eta}\cdot\|\partial_{t}\tilde{g}(t)\|^{1-\eta}_{W^{m^{'},2}_{t}},
\end{align*}
for some positive integer $m^{'}$. Note that we have
\begin{align*}
\|\partial_{t}\tilde{g}(t)\|^{1-\eta}_{W^{m^{'},2}_{t}}=2\|Ric(\tilde{g}(t))+\nabla^{2}f(t)-\frac{1}{2\tau}\tilde{g}(t)\|^{1-\eta}_{W^{m^{'},2}_{t}},
\end{align*}
but clearly, if $\delta>0$ is sufficiently small, then for all $t\in[0,T^{*}]$ we have
\begin{align}
\|\nabla^{2}\tilde{f}(t)\|_{W^{m^{'},2}_{t}} &\le C_{1}(m^{'})\label{regularbound}.
\end{align}
Observe that \eqref{regularbound} is a consequence of the fact that the metrics $\tilde{g}(t)$ are regular as in Definition \ref{regularmetric}. On the other hand and again by the estimates in \cite{ham82} we also have bounds of the form
\begin{align}
\|Rm(\tilde{g}(t))\|_{C^{m^{'}}_{t}}&\le C_{2}(m^{'}),\label{curvaturebound}
\end{align}
which are diffeomorphism-invariant. We then have
\begin{align*}
\|\partial_{t}\tilde{g}(t)\|_{W^{m,2}_{t}}\le C(m^{'})\|\partial_{t}\tilde{g}(t)\|^{\eta}_{L^{2}_{t}}.
\end{align*}
Our choice of the integer $m$ in \eqref{holderparameter} is such that we have
\begin{align}\label{estimateholder}
\|\partial_{t}\tilde{g}(t)\|_{C^{k,\gamma}_{t}}\le C\|\partial\tilde{g}(t)\|_{L^{2}_{t}}^{\eta},
\end{align}
where $C$ is a positive constant depending on $m^{'}$ and Sobolev constants (which are uniformly bounded in $[t_{0},t_{1}]$ since all metrics $\tilde{g}(t)$ are in $\U^{k,\gamma}_{\delta}$). Inequality \eqref{estimateholder} together with the {\L}ojasiewicz-Simon inequality in \ref{loja} will allow us to estimate the $C^{k,\gamma}$ norm of the difference $\tilde{g}(t)-\tilde{g}(t^{'})$ for $t,t^{'}\in [T_{0},T_{1}]$.  If we put 
\begin{align*}
\theta=-1+\eta+\alpha-\frac{\alpha\eta}{2}~\text{(note that}~\theta>0,
\end{align*}
then  for any $t\in(t_{0},t_{1})$ we have by Lemma \ref{lemmaloja} the differential inequality
\begin{align*}
\frac{d}{dt}(\mu(g(0),\tau)-\mu(\tilde{g}(t),\tau))^{\theta}\ge -C\theta\|\nabla\mu\|^{\eta}_{L^{2}_{t}(M)}, 
\end{align*}
and therefore, for any $\tau_{0}$,$\tau_{1}$ with $t_{0}\le \tau_{0}\le\tau_{1}\le t_{1}$ we obtain
\begin{align*}
&\|\tilde{g}(\tau_{1})-\tilde{g}(\tau_{0})\|_{C^{k,\gamma}_{t}}\le -(C\theta)^{-1}\int_{\tau_{0}}^{\tau_{1}}\frac{d}{ds}(\mu(g_{0})-\mu(\tilde{g}(s)))^{\theta}ds\\
&=(C\theta)^{-1}\left[\left(\mu(g_{0})-\mu(\tilde{g}(\tau_{0}))\right)^{\theta}-(\mu(g_{0})-\mu(\tilde{g}(\tau_{1})^{\theta}\right]\\
&\le(C\theta)^{-1}\left(\mu(g_{0})-\mu(\tilde{g}(t_{0}))\right)^{\theta}.
\end{align*}
Since the norms $\|\cdot\|_{C^{k,\gamma}_{t}(M)}$ and $\|\cdot\|_{C^{k,\gamma}_{g_{0}}(M)}$ are equivalent by Remark \ref{equivalent}, the Lemma follows.
\end{proof}

We are now ready to prove Theorem \ref{mainthm} as outlined in Section \ref{outline} of the introduction.

\begin{proof}[Proof Theorem \ref{mainthm}]Fix a regular neighborhood $\U^{k,\gamma}_{\delta}$ of $g_{0}$. Recall that from the results in \cite{sesum06} we can consider a sequence $t_{i}$ with $t_{i}\nearrow\infty$ as $i\rightarrow\infty$ such that $g(t_{i})\rightarrow g_{0}$ in $C^{k+4,\gamma}_{g(t_{i})}$, where $g(t)$ is a solution of \eqref{tauflow} which exists for all time. Fix $\epsilon>0$ and let $i_{0}$ be a positive integer large enough so that $g(t_{i})\in \U^{k+4,\gamma}_{\delta/8}$ for all $i\ge i_{0}$ and such that for $t\ge t_{i_{0}}$ we have $0\le \mu(g_{0},\tau)-\mu(g(t),\tau)\le \epsilon$  (this last condition is guaranteed by Lemma \ref{crucialmu}).  Let $T=t_{i_{0}}$, since $g(T)\in \U^{k+4,\gamma}_{\delta/8}$ is regular, we can consider the canonical solution $\tilde{g}(s)$ of the  modified Ricci flow \eqref{ModRF} with $\tilde{g}(0)=g(T)$. Consider an interval $[0,T_{0}]$ such that for all $s\in[0,T_{0}]$ the metrics $\tilde{g}(s)$  satisfy $\tilde{g}(s)\in\U^{k,\gamma}_{\delta/8}$ (see Lemma \ref{existence}). Note that for any $s\in[0,T_{0}]$, there exists a $C^{\infty}$ diffeomorphism $\phi_{s}:M\rightarrow M$ such that $\phi^{*}_{s}\tilde{g}(s)=g(T+s)$, and therefore by Lemma \ref{crucialmu} we have
\begin{align*}
 \mu(\tilde{g}(s),\tau)=\mu(g(T+s),\tau)\le \mu(g_{0},\tau).
 \end{align*}
 In particular, we have that $\mu(\tilde{g}(s),\tau)\le\mu(g_{0},\tau)$ for any time $s$ such that $\tilde{g}(s)$ is defined, and the difference $\mu(g_{0},\tau)-\mu(\tilde{g}(s),\tau)$ is nonincreasing in $s$. Let now $T^{*}>T_{0}$ be such that the solution $\tilde{g}(s)$ exists in $[0,T^{*}]$ and satisfies $\tilde{g}(s)\in\U^{k,\gamma}_{\delta/2}$. By Lemma \ref{stability}, for any $s\in[T_{0},T^{*}]$ we have 
\begin{align*}
\|\tilde{g}(s)-\tilde{g}(T_{0})\|_{C^{k,\gamma}_{g_{0}}}\le C\left[(\mu(g_{0},\tau)-\mu(\tilde{g}(s),\tau)\right]^{\theta}\le C(\epsilon)^{\theta},
\end{align*}
so we can choose $\epsilon$ to have $\|\tilde{g}(s)-\tilde{g}(T_{0})\|_{C^{k,\gamma}_{g_{0}}}\le \frac{\delta}{4}$. Let $s\in [0,T^{*}]$, if $s\in [0,T_{0}]$ then $\|\tilde{g}(s)-g_{0}\|_{C^{k,\gamma}_{g_{0}}}\le \delta/8$ by definition of $T_{0}$ and if $s\in [T_{0},T^{*}]$ we obtain
\begin{align}\label{threeeighths}
\|\tilde{g}(s)-g_{0}\|_{C^{k,\gamma}_{g_{0}}}\le\|\tilde{g}(s)-\tilde{g}(T_{0})\|_{C^{k,\gamma}_{g_{0}}}+\|\tilde{g}(T_{0})-g_{0}\|_{C^{k,\gamma}_{g_{0}}}\le \frac{\delta}{4}+\frac{\delta}{8}=\frac{3\delta}{8}.
\end{align}
Since $T^{*}$ is any number with  $T^{*}>T_{0}$ such that $\tilde{g}(s)\in\U^{k,\gamma}_{\delta/2}$ for $s\in [0,T^{*}]$, it follows that the canonical solution of the modified Ricci flow with initial condition $\tilde{g}(0)=g(T)$  exists for all $s\ge 0$ and $\tilde{g}(s)\in\U^{k,\gamma}_{\delta/2}$, because otherwise  the number
\begin{align*}
s_{0}=\sup\{s: \tilde{g}(\sigma)\in \U^{k,\gamma}_{\delta/2}~\text{for all}~\sigma\in[0,s]\},
\end{align*}
is finite and then from \eqref{threeeighths} we must have
\begin{align}
\limsup_{s\nearrow\ s_{0}}\|\tilde{g}(s)-g_{0}\|_{C^{k,\gamma}_{g_{0}}}<\frac{\delta}{2},
\end{align}
but then from Lemma \ref{criterion}, we can extend $\tilde{g}(s)$ to  an interval of the form $[0,s_{0}+\epsilon]$ for some $\epsilon>0$ and such that $\tilde{g}(s)\in\U^{k,\gamma}_{\delta/2}$ for all $s$ in $[0,s_{0}+\epsilon]$ which contradicts the maximality of $s_{0}$. We now know that the flow $\tilde{g}(s)$ exists for all $s\ge 0$ and therefore for any $s>0$ we have
\begin{align*}
\frac{d}{ds}\left(\mu(g_{0},\tau)-\mu(\tilde{g}(s),\tau)\right)^{\alpha-1}&=C(1-\alpha)\|\nabla\mu(\tilde{g}(s),\tau)\|^{2}_{L^{2}_{s}(M)}\left(\mu(g_{0},\tau)-\mu(\tilde{g}(s),\tau)\right)^{\alpha-2}\\
&\ge C(1-\alpha),
\end{align*}
where $C$ is the positive constant appearing in \eqref{loja}. We obtain for $s>0$ the inequality
\begin{align*}
\mu(g_{0},\tau)-\mu(\tilde{g}(s),\tau)\le Cs^{-\frac{1}{1-\alpha}}.
\end{align*}

From Lemma \ref{stability} it follows that for any $s,s^{'}>0$ we must have
\begin{align*}
\|\tilde{g}(s)-\tilde{g}(s^{'})\|_{C^{k,\gamma}_{g_{0}}}\le Cs^{-\frac{\theta}{1-\alpha}},
\end{align*}
where $\theta>0$ and from the equivalence of the norms $\|\cdot\|_{C^{k,\gamma}_{s}}$ and $\|\cdot\|_{C^{k,\gamma}_{g_{0}}}$ we conclude the existence of a metric $g_{\infty}\in\U^{k,\gamma}_{\delta/2}$ such that 
\begin{align*}
\|\tilde{g}(s)-g_{\infty}\|_{C^{k,\gamma}_{s}}\le C s^{-\frac{\theta}{1-\alpha}},~s>0.
\end{align*}
The metric $g_{\infty}$ is a shrinking gradient Ricci soliton since $\nabla\mu(g_{\infty})=0$. By construction and as pointed out at the beginning of the proof, there exists a diffeomorphism $\phi_{s}$ such that $\phi^{*}_{s}\tilde{g}(s)=g(T+s)$ where $g(t)$ is the solution of \eqref{tauflow} (which exists for all time by assumption) , so if for $t>T$ we write $t=T+s$ and $\phi^{'}_{t}=(\phi_{t-T})^{-1}$ we have
\begin{align*}
\|(\phi^{'})^{*}_{t}g(t)-g_{\infty}\|_{C^{k,\gamma}_{t}}\le C(t-T)^{-\frac{\theta}{1-\alpha}},~t\rightarrow\infty,
\end{align*}
and therefore
\begin{align*}
\|g(t)-\phi_{t-T}^{*}g_{\infty}\|_{C^{k,\gamma}_{g(t)}(M)}\le C(t-T)^{-\frac{\theta}{1-\alpha}},~t\rightarrow\infty,
\end{align*}
Since $\|g(t_{i})-g_{0}\|_{C^{k,\gamma}_{g(t_{i})}(M)}\rightarrow 0$ as $i\rightarrow\infty$, we conclude that
\begin{align*}
\lim_{i\rightarrow\infty}\|g_{0}-\phi^{*}_{t_{i}-T}g_{\infty}\|_{C^{k,\gamma}_{g(t_{i})}(M)}=0.
\end{align*}
It follows from the results in \cite[Chapter 10, Sections 3.2, 3.3, 3.4]{petersenbk} that there exists a diffeomorphism $\varphi:M\rightarrow M$ such that  $\varphi^{*}g_{\infty}=g_{0}$ which completes the proof.
\end{proof}

\appendix

\section{A note on the {\L}ojasiewicz-Simon inequality for the $\mu$-functional}\label{loj}

Even though a prove of the  {\L}ojasiewicz-Simon inequality \eqref{loja} was given in \cite{songwang10}, we include this section for showing an alternative way to compute the second variation of the $\mu$-functional at a shrinking gradient Ricci soliton and to relate \eqref{loja} to a more general {\L}ojasiewicz-Simon inequality proved in \cite{coldmini2012}. 
\subsection{Second Variation of the $\mu$-functional} 

In this subsection we compute the second variation of the $\mu$-functional at a shrinking gradient Ricci soliton. Recall that the metric $g_{0}$ is regular, and therefore there exists a regular $C^{k,\gamma}$ neighborhood $\U$ of $g_{0}$. For any $g\in\U$ let $P(g)$ denote the function in 
Convention \ref{conventionP}. As we will see below, it will be necessary for us to compute the linearization $P_{g_{0}}^{'}(h)$ for any $h\in S^{2}(T^{*}M)$. We introduce for this purpose the following notation used in \cite{songwang10}: let $(M,g)$ be a closed Riemannian manifold, let $f$ be a smooth function on $M$ and let $\nabla$ denote the covariant derivative with respect to $g$ . We define $(\nabla^{*})^{f}$ to the the dual of $\nabla$ with respect to the measure $e^{-f}dV_{g}$. In the same way,  we define $(d^{*})^{f}$ to be the dual of $d$ with respect to  the measure $e^{-f}dV_{g}$.  We also define the operators
\begin{align}
\Delta_{g}^{f}&=-(\nabla^{*})^{f}\nabla,\label{roughlaplacian}\\
\Delta^{f}_{H}&=d(d^{*})^{f}+(d^{*})^{f}d.\label{hodgelaplacian}
\end{align}
Note that the operators \eqref{roughlaplacian} \eqref{hodgelaplacian} can be defined on $p$-forms for any integer $p\ge 0$. A simple computation shows that if $\varphi$ is a smooth function then
\begin{align}\label{flaplacian}
\Delta_{g}^{f}\varphi=\Delta_{g}\varphi-\langle d\varphi,df\rangle_{g}.
\end{align}

Using a Weitzenb\"{o}ck-type formula relating $\Delta^{f}$ and $\Delta_{H}$ on 1-forms which was shown in \cite{perelman1} and integration by parts one can prove
\begin{lemma}\label{eigenvalue} If $g_{0}$ satisfies \eqref{solitoneq1} and $M$ is compact, the least positive eigenvalue of $-\Delta_{g_{0}}^{f_{0}}$ on functions is strictly greater than $\frac{1}{2\tau}$.
\footnote{This eigenvalue estimate is crucial in \cite{songwang10}  to prove Lemma \ref{lemregular} by means of the implicit function theorem. If $M$ is non-compact, then the least positive eigenvalue of $-\Delta^{f}_{g}$ is $\frac{1}{2\tau}$ if and only if $M$ splits off a line in which case eigenfunctions corresponding to $\frac{1}{2\tau}$ are linear functions as explained in \cite[Section 1.2]{heinnaber}.}  
\end{lemma} 

\begin{proof} See \cite{caozhu12}.
\end{proof}

In the case that $T$ is either a 1-form or a symmetric $(0,2)$ tensor , we use $\delta_{g}^{f}T$ to denote $-(\nabla^{*})^{f}T$. Suppose $h\in S^{2}(T^{*}M)$,  then

\begin{align}
\delta_{g}^{f}h&=\delta_{g} h-i_{\nabla f}h\label{divf},\\
\delta_{g}^{f}\delta_{g}^{f}&=\delta_{g}\delta_{g} h-2\langle\delta_{g} h,df\rangle_{g}-\langle h,\nabla_{g}^{2}f\rangle_{g}+\langle h,df\otimes df\rangle_{g}.\label{doubledivf}
\end{align}
where $i_{\nabla f}h$ denotes contraction of $h$ with respect to the vector field $\nabla f$. Observe that for all metrics $g$ in a regular neighborhood $\U$ of $g_{0}$ we must have
\begin{align}\label{equationP}
\Delta_{g} P(g)-\frac{1}{2}|\nabla P(g)|^{2}+\frac{1}{2}R_{g}+\frac{1}{2\tau}(P(g)-n)=\frac{1}{2\tau}\mu(g,\tau).
\end{align}
From \eqref{equationP} we have
\begin{lemma} Let $\Phi(h)=P^{'}_{g_{0}}(h)-\frac{1}{2}\tr_{g_{0}}(h)$, then $\Phi(h)$ is the unique solution of
\begin{align*}
\Delta_{g_{0}}\Phi(h)+\frac{1}{2\tau}\Phi(h)=-\frac{1}{2}\delta_{g_{0}}^{f_{0}}\delta_{g_{0}}^{f_{0}}h.
\end{align*}
In particular, if $\delta_{g_{0}}^{f_{0}}h=0$, we have $\Phi(h)\equiv 0$.
\end{lemma}

\begin{proof} Note that at  $g_{0}$ we have $\mu^{'}_{g_{0}}(h)=0$ for all $h\in S^{2}(T^{*}M)$, therefore, after taking linearizations in \eqref{equationP} we obtain 
\begin{align}\label{linearizedeq}
\left(\Delta_{g}P(g)-\frac{1}{2}|\nabla P(g)|^{2}+\frac{1}{2}R_{g}+\frac{1}{2\tau}(P(g)-n)\right)^{'}_{g_{0}}(h)=0,
\end{align}
for any $h\in S^{2}(T^{*}M)$. We now compute all the terms in \eqref{linearizedeq}:  
\begin{align}
\left(\Delta_{g}P(g)\right)^{'}_{g_{0}}(h)&=\Delta_{g_{0}}P^{'}_{g_{0}}(h)-\langle h,\nabla^{2}f_{0}\rangle-\langle
\delta_{g_{0}} h, df_{0}\rangle+\frac{1}{2}\langle d\tr_{g_{0}}(h),df_{0}\rangle,\label{laplacian}\\
\left(-\frac{1}{2}|\nabla P(g)|^{2}\right)^{'}_{g_{0}}(h)&=\frac{1}{2}\langle h,df_{0}\otimes df_{0}\rangle-\langle dP^{'}_{g_{0}}(h),d f_{0}\rangle,\label{gradsquare}\\
\frac{1}{2}R^{'}_{g_{0}}(h)&=-\frac{1}{2}\Delta_{g_{0}}\tr_{g_{0}}(h)+\frac{1}{2}\delta_{g_{0}}\delta_{g_{0}} h-\frac{1}{2}\langle Ric(g_{0}),h\rangle,\label{scalar}\\
\left(\frac{1}{2\tau}(P(g)-n)\right)^{'}_{g_{0}}(h)&=\frac{1}{2\tau}P^{'}_{g_{0}}(h).\label{zerothorder}
\end{align}
In \eqref{laplacian}-\eqref{zerothorder} all the inner products $\langle\cdot,\cdot\rangle$ are with respect to $g_{0}$. Using that $g_{0}$ is a shrinking gradient Ricci soliton we have 
\begin{align}\label{ricci}
-\frac{1}{2}\langle Ric(g_{0}),h\rangle=\frac{1}{2}\langle\nabla^{2}f_{0},h \rangle-\frac{1}{4\tau}\tr_{g_{0}}(h),
\end{align}
so combining \eqref{laplacian}-\eqref{zerothorder} with \eqref{ricci} and \eqref{flaplacian},\eqref{doubledivf} we obtain
\begin{align*}
\Delta^{f_{0}}\left(P^{'}_{g_{0}}(h)-\frac{1}{2}\tr_{g_{0}}(h)\right)+\frac{1}{2\tau}\left(P^{'}_{g_{0}}(h)-\frac{1}{2}\tr_{g_{0}}(h)\right)+\frac{1}{2}\delta_{g_{0}}^{f}\delta_{g_{0}}^{f}h=0.
\end{align*}
 as claimed.
 \end{proof}
 

We now can compute the second variation of $\mu(g,\tau)$ with respect to $g$ at $g_{0}$.

\begin{lemma}\label{secondvar} The second variation of $\mu(\cdot,\tau)$ at a soliton metric $g_{0}$ as in \eqref{solitoneq1} is given by an operator $L$ defined by
\begin{align*}
Lh=-\tau\left(\Delta_{L}h-\delta_{g_{0}}^{*}\delta_{g_{0}}h+\nabla^{2}_{g_{0}}\Phi(h)+\Psi(h)+\frac{1}{2\tau}h\right),
\end{align*}
where $\Delta_{L}$ is the Lichnerowicz laplacian on $(0,2)$ tensors with respect to $g_{0}$ and $\delta^{*}_{g_{0}}$ is the dual of $\delta_{g_{0}}$ with respect to the $L^{2}$ inner product associated to $dV_{g_{0}}$. Also,  $\Phi(h)$ is as before the unique solution of
\begin{align*}
\Delta^{f_{0}}_{g_{0}}\Phi(h)+\frac{1}{2\tau}\Phi(h)=-\frac{1}{2}\delta^{f_{0}}\delta^{f_{0}}h,
\end{align*}
and $\Psi(h)$ is given in coordinates by
\begin{align}\label{psicoordinates}
\Psi(h)_{ij}=-\frac{1}{2}g^{kl}\left(\nabla_{i}h_{jl}+\nabla_{j}h_{il}-\nabla_{l}h_{ij}\right)\nabla_{l}f_{0}.
\end{align}
In particular, if $\delta_{f_{0}}h=0$, it follows that $\Phi(h)\equiv 0$. 
\end{lemma}

\begin{remark}\em{ Our computation has some similarity with the results in \cite[Theorem 1.1]{caozhu12}, where the second variation of the $\nu$-functional defined by
\begin{align*}
\nu(g)=\inf_{\int_{M} d\nu=1,\tau>0}\W(g,f,\tau),
\end{align*}
is computed (here $d\nu=(4\pi\tau)^{-n/2}e^{-f_{0}}dV_{g}$). The above cited result was first obtained by Cao, Hamilton and Ilmanen in \cite{CHI} .}
\end{remark}

\begin{proof} The second variation of $\mu(\cdot,\tau)$ at  $g_{0}$ is a bilinear form $S^{2}(T^{*}M)\times S^{2}(T^{*}M)\rightarrow\R$ which on tensors $h_{1},h_{2}\in S^{2}(T^{*}M)$ is given by
\begin{align*}
\nabla^{2}\mu_{g_{0}}(h_{1},h_{2})=\frac{d}{ds}\langle\nabla \mu(g_{s}),h_{2}\rangle_{s}\big{|}_{s=0},
\end{align*}
where $\{g_{s}\}$ is a path of metrics satisfying $g_{s}\big{|}_{s=0}=g_{0}$ and $\frac{d}{ds}g_{s}\big{|}_{s=0}=h_{1}$ and $\langle\cdot,\cdot\rangle_{s}$ is the inner product defined as
\begin{align*}
\int_{M}\langle\cdot,\cdot\rangle_{g_{s}}d\nu_{s},
\end{align*}
and $d\nu_{s}=(4\pi\tau)^{-n/2}e^{-P(g_{s})}dV_{g_{s}}$. Therefore, we can identify $\nabla^{2}\mu_{g_{0}}$ with the operator
\begin{align*}
Lh=\frac{d}{ds}\nabla\mu(g_{s})\big{|}_{s=0}.
\end{align*} 
In other words, the second variation of $\mu$ at $g_{0}$ can be identified with the linearization of $\nabla\mu(g)$ at $g_{0}$. We therefore have   
\begin{align*}
Lh=\nabla\mu^{'}_{g_{0}}(h)=-\tau\left(Ric^{'}_{g_{0}}(h)-\frac{1}{2\tau}h+\left(\nabla^{2}P\right)^{'}_{g_{0}}(h)\right),
\end{align*}
where $\left(\nabla^{2}P\right)^{'}_{g_{0}}(h)$ is the linearization of $\nabla^{2}P(g)$ at $g_{0}$ in the direction of $h$. It is easy to see that
\begin{align*}
\left(\nabla^{2}P\right)^{'}_{g_{0}}(h)=\nabla^{2}_{g_{0}}\left(P^{'}_{g_{0}}(h)\right)+\Psi(h).
\end{align*}
where $\Psi(h)$ is given in coordinates by \eqref{psicoordinates}. The rest of the claim follows from Lemma \ref{linearizedeq} and from computing explicitly the linearization of the Ricci curvature at  $g_{0}$.
\end{proof}

Note that we have mentioned several times that if $\delta^{f_{0}}_{g_{0}}h=0$ then $\Phi(h)\equiv 0$. From this observation and the results in Lemma \ref{secondvar} we conclude that under the condition $\delta^{f_{0}}_{g_{0}}h=0$, the operator $L$ becomes \emph{elliptic} because for the second order $\nabla^{2}_{g_{0}}P^{'}_{g_{0}}(h)$ and $\delta_{g_{0}}^{*}\delta h$ we have 
\begin{align*}
\nabla^{2}_{g_{0}}P^{'}_{g_{0}}(h)&=0,\\
\delta_{g_{0}}^{*}\delta_{g_{0}} h&=\delta_{g_{0}}^{*}\left(i_{\nabla f_{0}}h\right).
\end{align*}
We now prove that it is always possible to fix a gauge where $\delta_{g_{0}}^{f_{0}}h=0$ based on an idea used in \cite{ct} (see also \cite[Lemma 2.10]{viagursky11} ).  Before stating the result, we let $\mathcal{D}(M)$ denote the group of $C^{\infty}$ diffeomorphisms of $M$.

\begin{lemma}[Gauge Fixing Lemma] \label{gaugefixing} Let $f\in C^{\infty}$ be a function and let $g$ be a Riemannian metric. There exists a $C^{k,\gamma}$ neighborhood  $\U$ of $g$ and a $C^{k+1,\gamma}$ neighborhood $\V$  of the identity map in $\mathcal{D}(M)$ such that for all $\tilde{g}\in \U$ there is a unique diffeomorphism $\phi_{\tilde{g}}:M\rightarrow M$ such that 
\begin{align*}
\delta^{f}_{g}\left(\phi_{\tilde{g}}^{*}\tilde{g}-g\right)=0.
\end{align*}
Moreover, the map $\U\rightarrow\V$ given by $g\rightarrow \phi_{\tilde{g}}$ is smooth and we have the estimate
\begin{align}\label{inverseestimate}
\|\phi_{\tilde{g}}^{*}\tilde{g}-g\|_{C^{k,\gamma}}\le C\|\tilde{g}-g\|_{C^{k,\gamma}}.
\end{align}
\end{lemma}

\begin{proof}
Let $\Pi:C^{k,\gamma}(M)\rightarrow C^{k,\gamma}(TM)$ denote the orthogonal projection onto the space of Killing fields for the metric $g$ with respect to the inner product
\begin{align}\label{weightip}
\int_{M}\langle\cdot,\cdot\rangle_{g} e^{-f}dV_{g}.
\end{align}
Let $X$ be a $C^{k,\gamma}$ vector field and let $\phi_{X,1}$ be the flow generated by $X$ at time 1. We consider the nonlinear map $\N:C^{k}(TM)\times C^{k}(\M(M))\rightarrow C^{k}(TM)$ given by
\begin{align*}
\N(X,\tilde{g})=\delta^{f}_{g}(\phi_{X,1}^{*}\tilde{g}-g)+\Pi(X).
\end{align*} 
Note that $\N(0,g)=0$ and the differential of $\N$ with respect to $X$ in the direction of $Y\in C^{k,\gamma}(TM)$ at $(0,g)$ is given by 
\begin{align*}
d\N_{(0,g)}(Y,\eta)=\frac{d}{dt}\left[\delta^{f}_{g}\left(\phi^{*}_{tY,1}g-g\right)+\Pi(tY)\right]\big{|}_{t=0}=\delta^{f}_{g}\mathcal{L}_{Y}g+\Pi(Y),
\end{align*}
where $\mathcal{L}_{Y}g$ is the Lie derivative operator with respect to $g$. The $L^{2}$ adjoint of $d\N_{(0,g,0)}$ with respect to the inner product \eqref{weightip}  is given by
\begin{align*}
\left(d\N_{(0,g)}\right)^{*}(Z)=\delta^{f}_{g}\mathcal{L}_{Z^{\#}}g+\Pi(Z^{\#}),
\end{align*}
where $Z^{\#}$ is the vector field dual to the 1-form $Z$ with respect to the metric $g$. Since $\Pi(Z^{\#})$ is a Killing field, it follows using integration by parts that elements in the kernel of $\left(d\N_{(0,g)}\right)^{*}$ satisfy
\begin{align}
\delta^{f}_{g}\mathcal{L}_{Z^{\#}}g&=0,\label{deltalie}\\
\Pi(Z^{\#})&=0\label{projkilling}.
\end{align}
Using integration by parts once more we see that equation \eqref{deltalie} implies that $Z^{\#}$ is a Killing field, but then equation \eqref{projkilling} shows that $Z^{\#}$ is identically zero so we conclude that the map $d\N_{(0,g)}$ is surjective. By the inverse function theorem, for $\tilde{g}$ sufficiently close to $g$ in $C^{k,\gamma}(S^{2}T^{*}M)$, we can find $X$ such that $\N(X,\tilde{g})=0$ which implies that 
\begin{align*}
\delta^{f}_{g}\left(\phi^{*}_{X,1}\tilde{g}-g\right)=-\Pi(X),
\end{align*} 
but since $\Pi(X)$ is a Killing field, we can argue as above to conclude that
\begin{align*}
\delta^{f}_{g}\left(\phi^{*}_{X,1}\tilde{g}-g\right)&=0,\\
\Pi(X)&=0.
\end{align*} 
The estimate \eqref{inverseestimate} follows easily from the inverse function theorem.

\end{proof}
Consider now the space $S^{2}_{*}\subset S^{2}$ given by
\begin{align*}
S^{2}_{*}=\{h\in S^{2}(T^{*}M):\delta^{f_{0}}_{g_{0}}h=0\}.
\end{align*}
Recall that we can identify the tangent space $T_{g_{0}}\M(M)$  with $S^{2}(T^{*}M)$. Lemma \ref{gaugefixing} implies that  there exists a small $C^{k,\gamma}$ neighborhood $\U$ of $0$  in $T_{g_{0}}\M(M)$ such that we can define a diffeomorphism $E_{g_{0}}:S^{2}_{*}\cap \U\rightarrow\V$ where $\V$ is a regular neighborhood of $g_{0}$. Instead of writing the map $E_{g_{0}}$ we write its inverse $E^{-1}_{g_{0}}:\U\rightarrow T_{g_{0}}\M\cap S^{2}_{*}$ which is given by 
\begin{align*}
E^{-1}_{g_{0}}(g)=\phi^{*}_{g}g-g_{0},
\end{align*}
where $\phi_{g}$ is the diffeomorphism obtained in Lemma \ref{gaugefixing}. A simple computation shows that the differential of $E_{g_{0}}$ at $0\in T_{g_{0}}\M$ is the identity map. Also, it is not difficult to prove that $E_{g_{0}}$ is a real analytic map. We now introduce a functional $G:S^{2}_{*}\cap\U\rightarrow\R$ given by
\begin{align*}
G(h)=\mu(E_{g_{0}}(h),\tau).
\end{align*}  
We can now apply the results in \cite[Section 6]{coldmini2012} to prove the following

\begin{lemma}\label{lojaG} There exists $0<\alpha<1$ such that 
\begin{align*}
|G(h)-G(0)|^{1-\frac{\alpha}{2}}\le C\|\nabla G(h)\|_{L^{2}_{h}(M)},
\end{align*}
for all $h\in S^{2}_{*}\cap\U$ if $\U$ is a sufficiently small $C^{k+1,\gamma}$ neighborhood of $0$ in $T_{g_{0}}\M(M)$. Here $L_{h}^{2}(M)$ is  the weighted $L^{2}$ space given by the measure $e^{-f_{h}}dV_{E_{g_{0}}(h)}$ and $f_{h}$ is the unique function satisfying $G(h)=\mu(E_{g_{0}}(h),\tau)=\W(E_{g_{0}}(h),f_{h},\tau)$. 
\end{lemma}
As shown also in \cite[Section 6]{coldmini2012}, Lemma \ref{lemmaloja} is a corollary of Lemma \ref{lojaG}.

\bibliography{entropyII}

\def\cprime{$'$}
\providecommand{\bysame}{\leavevmode\hbox to3em{\hrulefill}\thinspace}
\providecommand{\MR}{\relax\ifhmode\unskip\space\fi MR }
\providecommand{\MRhref}[2]{%
  \href{http://www.ams.org/mathscinet-getitem?mr=#1}{#2}
}
\providecommand{\href}[2]{#2}
\begin{thebibliography}{HDHI04}

\bibitem[CK04]{chowknopf}
Bennett Chow and Dan Knopf, \emph{The {R}icci flow: an introduction},
  Mathematical Surveys and Monographs, vol. 110, American Mathematical Society,
  Providence, RI, 2004. \MR{2061425 (2005e:53101)}

\bibitem[CM12]{coldmini2012}
Tobias~H. Colding and William~P. Minicozzi, \emph{On uniqueness of tangent
  cones for {E}instein manifolds}, arXiv:1206.4929, 2012.

\bibitem[CT94]{ct}
Jeff Cheeger and Gang Tian, \emph{On the cone structure at infinity of {R}icci
  flat manifolds with {E}uclidean volume growth and quadratic curvature decay},
  Invent. Math. \textbf{118} (1994), no.~3, 493--571.

\bibitem[CZ12]{caozhu12}
Huai-Dong Cao and Meng Zhu, \emph{On second variation of {P}erelman's {R}icci
  shrinker entropy}, Math. Ann. \textbf{353} (2012), no.~3, 747--763.
  \MR{2923948}

\bibitem[GV11]{viagursky11}
Matthew Gursky and Jeff Viaclovsky, \emph{Rigidity and stability of {E}instein
  metrics for quadraric curvature functionals}, arXiv:1105.4648, 2011.

\bibitem[Ham82]{ham82}
Richard~S. Hamilton, \emph{Three-manifolds with positive {R}icci curvature}, J.
  Differential Geom. \textbf{17} (1982), no.~2, 255--306. \MR{664497
  (84a:53050)}

\bibitem[Has12]{hasl2012}
Robert Haslhofer, \emph{Perelman's lambda-functional and the stability of
  {R}icci-flat metrics}, Calc. Var. Partial Differential Equations \textbf{45}
  (2012), no.~3-4, 481--504. \MR{2984143}

\bibitem[HDHI04]{CHI}
Cao Huai-Dong, Richard~S. Hamilton, and Tom Ilmanen, \emph{Gaussian densities
  and stability for some {R}icci solitons}, arXiv:math/0404165, 2004.

\bibitem[HM13]{muhas2013}
Robert Haslhofer and Reto M\"{u}ller, \emph{Dynamical stability and instability
  of {R}icci-flat metrics}, arXiv:1301.3219, 2013.

\bibitem[HN12]{heinnaber}
Hans-Joachim Hein and Aaron Naber, \emph{New logarithmic sobolev inequalities
  and an $\epsilon$-regularity theorem for the ricci flow}, arXiv:1205.0380
  ({T}o appear in {CPAM}), 2012.

\bibitem[KL06]{kleinerlott}
Bruce Kleiner and John. Lott, \emph{Notes on {P}erelman's papers},
  arXiv:math/0605667, 2006.

\bibitem[Nir01]{nirenbergtopics}
Louis Nirenberg, \emph{Topics in nonlinear functional analysis}, Courant
  Lecture Notes in Mathematics, vol.~6, New York University Courant Institute
  of Mathematical Sciences, New York, 2001, Chapter 6 by E. Zehnder, Notes by
  R. A. Artino, Revised reprint of the 1974 original. \MR{1850453
  (2002j:47085)}

\bibitem[Pal12a]{npalisrf}
Nefton Pali, \emph{The soliton {R}icci flow on compact manifolds},
  arXiv:1203.3682, 2012.

\bibitem[Pal12b]{npaliwvar}
\bysame, \emph{The total second variation of {P}erelman's
  $\mathcal{W}$-functional}, arXiv:1201.0969, 2012.

\bibitem[Per02]{perelman1}
Grisha Perelman, \emph{The entropy formula for the {R}icci {F}low and its
  {G}eometric {A}pplications}, arXiv:math/0211159, 2002.

\bibitem[Pet06]{petersenbk}
Peter Petersen, \emph{Riemannian geometry}, second ed., Graduate Texts in
  Mathematics, vol. 171, Springer, New York, 2006. \MR{2243772 (2007a:53001)}

\bibitem[Rot81]{rothaus}
O.~S. Rothaus, \emph{Logarithmic {S}obolev inequalities and the spectrum of
  {S}chr\"odinger operators}, J. Funct. Anal. \textbf{42} (1981), no.~1,
  110--120. \MR{620582 (83f:58080b)}

\bibitem[Sch11]{schulze11}
Felix Schulze, \emph{Uniqueness of compact tangent flows in {M}ean {C}urvature
  {F}low}, arXiv:1107.4643, 2011.

\bibitem[Ses06]{sesum06}
Natasa Sesum, \emph{Convergence of the {R}icci flow toward a soliton}, Comm.
  Anal. Geom. \textbf{14} (2006), no.~2, 283--343. \MR{2255013 (2007f:53080)}

\bibitem[Sim83]{ls83}
Leon Simon, \emph{Asymptotics for a class of nonlinear evolution equations,
  with applications to geometric problems}, Ann. of Math. (2) \textbf{118}
  (1983), no.~3, 525--571. \MR{727703 (85b:58121)}

\bibitem[Sim96]{lsimonbk2}
\bysame, \emph{Theorems on regularity and singularity of energy minimizing
  maps}, Lectures in Mathematics ETH Z\"urich, Birkh\"auser Verlag, Basel,
  1996, Based on lecture notes by Norbert Hungerb{\"u}hler. \MR{1399562
  (98c:58042)}

\bibitem[SW12]{songwang10}
Song Sun and Yuanqi Wang, \emph{On the {K}\"{a}hler-{R}icci flow near a
  {K}\"{a}hler-{E}instein metric}, arXiv:1004.2018 ({T}o appear in {J}. reine
  angew. {M}ath.), 2012.

\end{thebibliography}
\end{document}